\documentclass{amsart}
\usepackage{amsfonts}

\newtheorem{theorem}{Theorem}
\theoremstyle{plain}

\newtheorem{corollary}{Corollary}

\newtheorem{definition}{Definition}

\newtheorem{lemma}{Lemma}

\newtheorem{proposition}{Proposition}
\newtheorem{remark}{Remark}

\numberwithin{equation}{section}
\input{tcilatex}

\begin{document}
\title[The Hermite -Hadamard's inequalities for preinvex functions]{On
Hermite-Hadamard Type Integral Inequalities for preinvex and log-preinvex
functions }
\author{Mehmet Zeki SARIKAYA}
\address{Department of Mathematics, \ Faculty of Science and Arts, D\"{u}zce
University, D\"{u}zce-TURKEY}
\email{sarikayamz@gmail.com}
\author{Hakan Bozkurt}
\email{insedi@yahoo.com}
\author{Necmettin Alp}
\email{placenn@gmail.com}

\begin{abstract}
In this paper, we extend some estimates of the left hand side of a Hermite-
Hadamard type inequality for nonconvex functions whose derivatives absolute
values are preinvex and log-preinvex.
\end{abstract}

\maketitle

\section{Introduction}

The following inequality is well-known in the literature as Hermite-Hadamard
inequality: Let $f:I\subset 
\mathbb{R}
\rightarrow 
\mathbb{R}
$ be a convex function on an interval $I$ of real numbers and $a,b\in I$
with $a<b$. Then the following holds%
\begin{equation}
f\left( \frac{a+b}{2}\right) \leq \frac{1}{b-a}\int\limits_{a}^{b}f\left(
x\right) dx\leq \frac{f\left( a\right) +f\left( b\right) }{2}.  \label{h}
\end{equation}%
Both inequalities hold in the reversed direction if the function $f$ is
concave.

The inequalities (\ref{h}) have become an important cornerstone in
mathematical analysis and optimization and many uses of these inequalities
have been discovered in a variety of settings. Recently, Hermite-Hadamard
type inequality has been the subject of intensive research. For recent
results, refinements, counterparts, generalizations and new Hadamard's-type
inequalities, see (\cite{SSDRPA}, \cite{Dragomir2}, \cite{USK}-\cite{CEMPJP}%
, \cite{sarikaya}-\cite{sarikaya5}).

In \cite{USK} some inequalities of Hermite-Hadamard type for differentiable
convex mappings connected with the left part of (\ref{h}) were proved using
the following lemma:

\begin{lemma}
\label{l1} Let $f:I^{\circ }\subset \mathbb{R}\rightarrow \mathbb{R}$, be a
differentiable mapping on $I^{\circ }$, $a,b\in I^{\circ }$ ($I^{\circ }$ is
the interior of $I$) with $a<b$. If $f^{\prime }\in L\left( \left[ a,b\right]
\right) $, then we have%
\begin{equation}
\begin{array}{l}
\dfrac{1}{b-a}\dint_{a}^{b}f(x)dx-f\left( \dfrac{a+b}{2}\right) \\ 
\\ 
\ \ \ \ \ =\left( b-a\right) \left[ \dint_{0}^{\frac{1}{2}}tf^{\prime
}(ta+(1-t)b)dt+\dint_{\frac{1}{2}}^{1}\left( t-1\right) f^{\prime
}(ta+(1-t)b)dt\right] .%
\end{array}
\label{HH}
\end{equation}
\end{lemma}

One more general result related to (\ref{HH}) was established in \cite%
{USKMEO}. The main result in \cite{USK} is as follows:

\begin{theorem}
\label{t1} Let $f:I\subset \mathbb{R}\rightarrow \mathbb{R}$, be a
differentiable mapping on $I^{\circ }$, $a,b\in I$ with $a<b$. If the
mapping $\left\vert f^{\prime }\right\vert $ is convex on $\left[ a,b\right] 
$, then%
\begin{equation}
\left\vert \frac{1}{b-a}\int_{a}^{b}f(x)dx-f\left( \frac{a+b}{2}\right)
\right\vert \leq \frac{b-a}{4}\left( \frac{\left\vert f^{\prime
}(a)\right\vert +\left\vert f^{\prime }(b)\right\vert }{2}\right) .
\label{H1}
\end{equation}
\end{theorem}

It is well known that convexity has been playing a key role in mathematical
programming, engineering, and optimization theory. The generalization of
convexity is one of the most important aspects in mathematical programming
and optimization theory. There have been many attempts to weaken the
convexity assumptions in the literature, (see, \cite{SSDRPA}, \cite%
{Dragomir2}, \cite{USK}-\cite{CEMPJP}, \cite{sarikaya}-\cite{sarikaya5}). A
significant generalization of convex functions is that of invex functions
introduced by Hanson in \cite{Hanson}. Ben-Israel and Mond \cite{IsraelMond}
introduced the concept of preinvex functions, which is a special case of
invexity. Pini \cite{Pini} introduced the concept of prequasiinvex functions
as a generalization of invex functions. Noor \cite{Noor0}-\cite{Noor2} has
established some Hermite-Hadamard type inequalities for preinvex and
log-preinvex functions. In recent papers Barani, Ghazanfari, and Dragomir in 
\cite{BGDragomir} presented some estimates of the right hand side of a
Hermite- Hadamard type inequality in which some preinvex functions are
involved. His class of nonconvex functions include the classical convex
functions and its various classes as special cases. For some recent results
related to this nonconvex functions, see the papers (\cite{Yangdili}-\cite%
{Noor2}, \cite{Hanson}-\cite{Pini}).

\section{Preliminaries}

Let $f:K\rightarrow 
\mathbb{R}
$, and $\eta (.,.):K\times K\rightarrow 
\mathbb{R}
$ , where $K$ is a nonempty closed set in $%
\mathbb{R}
^{n}$, be continuous functions. First of all, we recall the following well
known results and concepts, see \cite{Yangdili}-\cite{Noor2} \cite{mohan}
and the references theirin

\begin{definition}
\label{d1} Let $u\in K$. Then the set $K$ is said to be invex at $u$ with
respect to $\eta (.,.)$, if%
\begin{equation*}
u+t\eta (v,u)\in K,\text{ }\forall u,v\in K,\text{ }t\in \left[ 0,1\right] .
\end{equation*}%
$K$ is said to be an invex set with respect to $\eta $, if $K$ is invex at
each $u\in K$. The invex set $K$ is also called $\eta $-connected set.

\begin{remark}
\label{r1} We would like to mention that the Definition \ref{d1} of an invex
set has a clear geometric interpretation. This definition essentially says
that there is a path starting from a point $u$ which is contained in $K$. We
do not require that the point $v$ should be one of the end points of the
path. This observation plays an important role in our analysis. Note that,
if we demand that $v$ should be an end point of the path for every pair of
points, $u,v\in K$, then $\eta (v,u)=v-u$ and consequently invexity reduces
to convexity. Thus, it is true that every convex set is also an \.{I}nvex
set with respect to $\eta (v,u)=v-u$, but the converse is not necessarily
true, see \cite{Yangdili}-\cite{Noor2} and the references therein.
\end{remark}
\end{definition}

\begin{definition}
\label{d2} The function $f$ on the invex set $K$ is said to be preinvex with
respect to $\eta $, if%
\begin{equation*}
f\left( u+t\eta (v,u)\right) \leq \left( 1-t\right) f\left( u\right)
+tf\left( v\right) ,\text{ }\forall u,v\in K,\text{ }t\in \left[ 0,1\right] .
\end{equation*}%
The function $f$ is said to be preconcave if and only if $-f$ is preinvex.
Note that every convex function is an preinvex function, but the converse is
not true.
\end{definition}

\begin{definition}
\label{d3} The function $f$ on the invex set $K$ is said to be logarithmic
preinvex with respect to $\eta $, such that%
\begin{equation*}
f\left( u+t\eta (v,u)\right) \leq \left( f\left( u\right) \right)
^{1-t}\left( f\left( v\right) \right) ^{t},\text{ }u,v\in K,\text{ }t\in %
\left[ 0,1\right]
\end{equation*}%
where $f\left( .\right) >0$.
\end{definition}

Now we define a new definition for prequasiinvex functions as follows:

\begin{definition}
\label{d4} The function $f$ on the invex set $K$ is said to be prequasiinvex
with respect to $\eta $, if%
\begin{equation*}
f\left( u+t\eta (v,u)\right) \leq \max \left\{ f\left( u\right) ,f\left(
v\right) \right\} ,\text{ }u,v\in K,\text{ }t\in \left[ 0,1\right] .
\end{equation*}
\end{definition}

From the above definitions, we have%
\begin{eqnarray*}
f\left( u+t\eta (v,u)\right)  &\leq &\left( f\left( u\right) \right)
^{1-t}\left( f\left( v\right) \right) ^{t} \\
&\leq &\left( 1-t\right) f\left( u\right) +tf\left( v\right)  \\
&\leq &\max \left\{ f\left( u\right) ,f\left( v\right) \right\} .
\end{eqnarray*}%
We also need the following assumption regarding the function $\eta $ which
is due to Mohan and Neogy \cite{mohan}:

\textbf{Condition C} Let $K\subseteq 
\mathbb{R}
$ be an open invex subset with respect to $\eta :K\times K\rightarrow 
\mathbb{R}
$. For any $x,y\in K$ and any $t\in \left[ 0,1\right] ,$%
\begin{eqnarray*}
\eta (y,y+t\eta (x,y)) &=&-t\eta (x,y) \\
\eta (x,y+t\eta (x,y)) &=&\left( 1-t\right) \eta (x,y).
\end{eqnarray*}%
Note that for every $x,y\in K$ and every $t_{1},t_{2}\in \left[ 0,1\right] $
from Condition C, we have%
\begin{equation}
\eta (y+t_{2}\eta (x,y),y+t_{1}\eta (x,y))=\left( t_{2}-t_{1}\right) \eta
(x,y).  \label{c}
\end{equation}%
In \cite{Noor0},\ Noor proved the Hermite-Hadamard inequality for the
preinvex functions as follows:

\begin{theorem}
\label{t5} Let $f:K=\left[ a,a+\eta (b,a)\right] \rightarrow \left( 0,\infty
\right) $ be an preinvex\ function on the interval of real numbers $K^{0}$
(the interior of $K$) and $a,b\in K^{0}$ with $a<a+\eta (b,a)$. Then the
following inequality holds:%
\begin{equation}
f\left( \frac{2a+\eta (b,a)}{2}\right) \leq \frac{1}{\eta (b,a)}%
\dint\limits_{a}^{a+\eta (b,a)}f\left( x\right) dx\leq \frac{f(a)+f(a+\eta
(b,a))}{2}\leq \frac{f\left( a\right) +f\left( b\right) }{2}.  \label{p}
\end{equation}
\end{theorem}

In \cite{BGDragomir}, Barani, Gahazanfari and Dragomir proved the following
theorems.

\begin{theorem}
\label{t6} Let $A\subseteq 
\mathbb{R}
$ be an open invex subset with respect to $\eta :A\times A\rightarrow 
\mathbb{R}
$. Suppose that $f:A\rightarrow 
\mathbb{R}
$ is a diferentiable function.Assume $p\in 
\mathbb{R}
$ with $p>1$. If $\left\vert f^{\prime }\right\vert ^{\frac{p}{p-1}}$ is
prequasiinvex on $A$ then, for every $a,b\in A$ the following inequality
holds%
\begin{eqnarray*}
&&\left\vert \frac{f(a)+f(a+\eta (b,a))}{2}-\frac{1}{\eta (a,b)}%
\int_{b}^{b+\eta (b,a)}f(x)dx\right\vert  \\
&& \\
&\leq &\frac{\eta (b,a)}{2(p+1)^{\frac{1}{p}}}\left[ \sup \left\{ \left\vert
f^{\prime }\left( a\right) \right\vert ^{\frac{p}{p-1}},\left\vert f^{\prime
}\left( b\right) \right\vert ^{\frac{p}{p-1}}\right\} \right] ^{\frac{p}{p-1}%
}
\end{eqnarray*}
\end{theorem}

\begin{theorem}
\label{t7} Let $A\subseteq 
\mathbb{R}
$ be an open invex subset with respect to $\eta :A\times A\rightarrow 
\mathbb{R}
$. Suppose that $f:A\rightarrow 
\mathbb{R}
$ is a diferentiable function. If $\left\vert f^{\prime }\right\vert $ is
prequasiinvex on $A$ then, for every $a,b\in A$ the following inequality
holds%
\begin{eqnarray*}
&&\left\vert \frac{f(a)+f(a+\eta (b,a))}{2}-\frac{1}{\eta (a,b)}%
\int_{b}^{b+\eta (b,a)}f(x)dx\right\vert  \\
&& \\
&\leq &\frac{\eta (b,a)}{4}\max \left\{ \left\vert f^{\prime }\left(
a\right) \right\vert ,\left\vert f^{\prime }\left( b\right) \right\vert
\right\} 
\end{eqnarray*}
\end{theorem}

In this article, using functions whose derivatives absolute values are
preinvex and log-preinvex, we obtained new inequalities releted to the left
side of Hermite-Hadamard inequality for nonconvex functions.

\section{Hermite-Hadamard type inequalities for preinvex functions}

We shall start with the following refinements of the Hermite-Hadamard
inequality for preinvex functions. Firstly, we give the following results
connected with the left part of (\ref{p}):

\begin{theorem}
\label{tt3} Let $K\subseteq 
\mathbb{R}
$ be an open invex subset with respect to $\eta :K\times K\rightarrow 
\mathbb{R}
$. Suppose that $f:K\rightarrow 
\mathbb{R}
$ is a diferentiable function. If $\left\vert f^{\prime }\right\vert $ is
preinvex on $K,$ then, for every $a,b\in K$ the following inequality holds:%
\begin{equation}
\left\vert \frac{1}{\eta (b,a)}\int_{a}^{a+\eta (b,a)}f(x)dx-f\left( \frac{%
2a+\eta (b,a)}{2}\right) \right\vert \leq \frac{\eta (b,a)}{8}\left[
\left\vert f^{\prime }(a)\right\vert +\left\vert f^{\prime }(b)\right\vert %
\right]  \label{4}
\end{equation}

\begin{proof}
Suppose that $a,a+\eta (b,a)$ $\in K$. Since $K$ is invex with respect to $%
\eta $, for every $t\in \left[ 0,1\right] $, we have $a+\eta (b,a)\in K$.
Integrating by parts implies that%
\begin{eqnarray}
&&\int_{0}^{\frac{1}{2}}tf^{\prime }(a+t\eta (b,a))dt+\int_{\frac{1}{2}%
}^{1}(t-1)f^{\prime }(a+t\eta (b,a))dt  \notag \\
&&  \label{5} \\
&=&\left[ \frac{tf(a+t\eta (b,a))}{\eta (b,a)}\right] _{0}^{\frac{1}{2}}+%
\left[ \frac{(t-1)f(a+t\eta (b,a))}{\eta (b,a)}\right] _{\frac{1}{2}}^{1} 
\notag \\
&&-\frac{1}{\eta (b,a)}\int_{0}^{1}f(a+t\eta (b,a))dt  \notag \\
&&  \notag \\
&=&\frac{1}{\eta (b,a)}f\left( \frac{2a+\eta (b,a)}{2}\right) -\frac{1}{%
\left[ \eta (b,a)\right] ^{2}}\int_{a}^{a+\eta (b,a)}f(x)dx.  \notag
\end{eqnarray}%
By preinvex function of \ $\left\vert f^{\prime }\right\vert $ and (\ref{5}%
), we have%
\begin{eqnarray*}
&&\left\vert \frac{1}{\eta (b,a)}\int_{a}^{a+\eta (b,a)}f(x)dx-f\left( \frac{%
2a+\eta (b,a)}{2}\right) \right\vert \\
&\leq &\eta (b,a)\left[ \int_{0}^{\frac{1}{2}}t\left\vert f^{\prime
}(a+t\eta (b,a))\right\vert dt+\int_{\frac{1}{2}}^{1}(1-t)\left\vert
f^{\prime }(a+t\eta (b,a))\right\vert dt\right] \\
&\leq &\eta (b,a)\left[ \int_{0}^{\frac{1}{2}}t\left[ (1-t)\left\vert
f^{\prime }(a)\right\vert +t\left\vert f^{\prime }(b)\right\vert \right]
dt+\int_{\frac{1}{2}}^{1}(1-t)\left[ (1-t)\left\vert f^{\prime
}(a)\right\vert +t\left\vert f^{\prime }(b)\right\vert \right] dt\right] \\
&\leq &\eta (b,a)\left[ \frac{\left\vert f^{\prime }(a)\right\vert
+\left\vert f^{\prime }(b)\right\vert }{8}\right] .
\end{eqnarray*}%
The proof is completed.
\end{proof}
\end{theorem}

\begin{theorem}
\label{tt4} Let $K\subseteq 
\mathbb{R}
$ be an open invex subset with respect to $\eta :K\times K\rightarrow 
\mathbb{R}
$. Suppose that $f:K\rightarrow 
\mathbb{R}
$ is a diferentiable function. Assume $p\in 
\mathbb{R}
$ with $p>1$. If $\left\vert f^{\prime }\right\vert ^{\frac{p}{p-1}}$ is
preinvex on $K$ then, for every $a,b\in K$ the following inequality holds%
\begin{eqnarray}
&&\left\vert \frac{1}{\eta (b,a)}\int_{a}^{a+\eta (b,a)}f(x)dx-f\left( \frac{%
2a+\eta (b,a)}{2}\right) \right\vert  \notag \\
&&  \label{6} \\
&\leq &\frac{\eta (b,a)}{16}\left( \frac{4}{p+1}\right) ^{\frac{1}{p}}\left[
\left( 3\left\vert f^{\prime }(a)\right\vert ^{\frac{p}{p-1}}+\left\vert
f^{\prime }(b)\right\vert ^{\frac{p}{p-1}}\right) ^{\frac{p-1}{p}}+\left(
\left\vert f^{\prime }(a)\right\vert ^{\frac{p}{p-1}}+3\left\vert f^{\prime
}(b)\right\vert ^{\frac{p}{p-1}}\right) ^{\frac{p-1}{p}}\right] .  \notag
\end{eqnarray}

\begin{proof}
Suppose that $a,a+\eta (b,a)\in K$. By assumption, H\"{o}lder's inequality
and (\ref{5}) in the proof of Theorem \ref{tt3}, we have%
\begin{eqnarray*}
&&\left\vert \frac{1}{\eta (b,a)}\int_{a}^{a+\eta (b,a)}f(x)dx-f\left( \frac{%
2a+\eta (b,a)}{2}\right) \right\vert \\
&\leq &\eta (b,a)\left[ \int_{0}^{\frac{1}{2}}t\left\vert f^{\prime
}(a+t\eta (b,a))\right\vert dt+\int_{\frac{1}{2}}^{1}(1-t)\left\vert
f^{\prime }(a+t\eta (b,a))\right\vert dt\right] \\
&\leq &\eta (b,a)\left[ \left( \int_{0}^{\frac{1}{2}}t^{p}dt\right) ^{\frac{1%
}{p}}\left( \int_{0}^{\frac{1}{2}}\left\vert f^{\prime }(a+t\eta
(b,a))\right\vert ^{\frac{p}{p-1}}dt\right) ^{\frac{p-1}{p}}\right. \\
&&\left. +\left( \int_{\frac{1}{2}}^{1}\left( 1-t\right) ^{p}dt\right) ^{%
\frac{1}{p}}\left( \int_{\frac{1}{2}}^{1}\left\vert f^{\prime }(a+t\eta
(b,a))\right\vert ^{\frac{p}{p-1}}dt\right) ^{\frac{p-1}{p}}\right] \\
&\leq &\frac{\eta (b,a)}{2^{1+\frac{1}{p}}(p+1)^{\frac{1}{p}}}\left[ \left(
\int_{0}^{\frac{1}{2}}\left[ (1-t)\left\vert f^{\prime }(a)\right\vert ^{%
\frac{p}{p-1}}+t\left\vert f^{\prime }(b)\right\vert ^{\frac{p}{p-1}}\right]
dt\right) ^{\frac{p-1}{p}}\right. \\
&&\left. +\left( \int_{\frac{1}{2}}^{1}\left[ (1-t)\left\vert f^{\prime
}(a)\right\vert ^{\frac{p}{p-1}}+t\left\vert f^{\prime }(b)\right\vert ^{%
\frac{p}{p-1}}\right] dt\right) ^{\frac{p-1}{p}}\right] \\
&=&\frac{\eta (b,a)}{16}\left( \frac{4}{p+1}\right) ^{\frac{1}{p}}\left[
\left( 3\left\vert f^{\prime }(a)\right\vert ^{\frac{p}{p-1}}+\left\vert
f^{\prime }(b)\right\vert ^{\frac{p}{p-1}}\right) ^{\frac{p-1}{p}}+\left(
\left\vert f^{\prime }(a)\right\vert ^{\frac{p}{p-1}}+3\left\vert f^{\prime
}(b)\right\vert ^{\frac{p}{p-1}}\right) ^{\frac{p-1}{p}}\right]
\end{eqnarray*}%
which completes the proof.
\end{proof}
\end{theorem}

\begin{theorem}
\label{tt5} Under the assumptaions of Theorem \ref{tt4}. Then, for every $%
a,b\in K$ the following inequality holds%
\begin{eqnarray}
&&\left\vert \frac{1}{\eta (b,a)}\int_{a}^{a+\eta (b,a)}f(x)dx-f\left( \frac{%
2a+\eta (b,a)}{2}\right) \right\vert  \notag \\
&&  \label{7} \\
&\leq &\frac{\eta (b,a)}{16}\left( \frac{4}{p+1}\right) ^{\frac{1}{p}}(3^{%
\frac{p-1}{p}}+1)\left[ \left\vert f^{\prime }(a)\right\vert +\left\vert
f^{\prime }(b)\right\vert \right] .  \notag
\end{eqnarray}

\begin{proof}
We consider the inequality (\ref{6}) i.e.%
\begin{eqnarray*}
&&\left\vert \frac{1}{\eta (b,a)}\int_{a}^{a+\eta (b,a)}f(x)dx-f\left( \frac{%
2a+\eta (b,a)}{2}\right) \right\vert \\
&& \\
&\leq &\frac{\eta (b,a)}{16}\left( \frac{4}{p+1}\right) ^{\frac{1}{p}}\left[
\left( 3\left\vert f^{\prime }(a)\right\vert ^{\frac{p}{p-1}}+\left\vert
f^{\prime }(b)\right\vert ^{\frac{p}{p-1}}\right) ^{\frac{p-1}{p}}+\left(
\left\vert f^{\prime }(a)\right\vert ^{\frac{p}{p-1}}+3\left\vert f^{\prime
}(b)\right\vert ^{\frac{p}{p-1}}\right) ^{\frac{p-1}{p}}\right] .
\end{eqnarray*}%
Let $a_{1}=3\left\vert f^{\prime }(a)\right\vert ^{\frac{p}{p-1}}$, $%
b_{1}=\left\vert f^{\prime }(b)\right\vert ^{\frac{p}{p-1}}$, $%
a_{2}=\left\vert f^{\prime }(a)\right\vert ^{\frac{p}{p-1}}$, $%
b_{2}=3\left\vert f^{\prime }(b)\right\vert ^{\frac{p}{p-1}}$. Here $%
0<\left( p-1\right) /p<1$, for $p>1$. Using the fact that,%
\begin{equation*}
\dsum\limits_{k=1}^{n}\left( a_{k}+b_{k}\right) ^{s}\leq
\sum_{k=1}^{n}a_{k}^{s}+\sum_{k=1}^{n}b_{k}^{s}
\end{equation*}%
For $\left( 0\leq s<1\right) $, $a_{1},a_{2},...,a_{n}\geq 0$, $%
b_{1},b_{2},...,b_{n}\geq 0$, we obtain%
\begin{eqnarray*}
&&\frac{\eta (b,a)}{16}\left( \frac{4}{p+1}\right) ^{\frac{1}{p}}\left[
\left( 3\left\vert f^{\prime }(a)\right\vert ^{\frac{p}{p-1}}+\left\vert
f^{\prime }(b)\right\vert ^{\frac{p}{p-1}}\right) ^{\frac{p-1}{p}}+\left(
\left\vert f^{\prime }(a)\right\vert ^{\frac{p}{p-1}}+3\left\vert f^{\prime
}(b)\right\vert ^{\frac{p}{p-1}}\right) ^{\frac{p-1}{p}}\right] \\
&\leq &\frac{\eta (b,a)}{16}\left( \frac{4}{p+1}\right) ^{\frac{1}{p}}(3^{%
\frac{p-1}{p}}+1)\left[ \left\vert f^{\prime }(a)\right\vert +\left\vert
f^{\prime }(b)\right\vert \right]
\end{eqnarray*}%
which completed proof.
\end{proof}
\end{theorem}

\begin{theorem}
\label{tt6} Let $K\subseteq 
\mathbb{R}
$ be an open invex subset with respect to $\eta :K\times K\rightarrow 
\mathbb{R}
$. Suppose that $f:K\rightarrow 
\mathbb{R}
$ is a diferentiable function. Assume $q\in 
\mathbb{R}
$ with $q\geq 1$. If $\left\vert f^{\prime }\right\vert ^{q}$ is preinvex on 
$K$ then, for every $a,b\in K$ the following inequality holds%
\begin{eqnarray}
&&\left\vert \frac{1}{\eta (b,a)}\int_{a}^{a+\eta (b,a)}f(x)dx-f\left( \frac{%
2a+\eta (b,a)}{2}\right) \right\vert   \notag \\
&&  \label{8} \\
&\leq &\frac{\eta (b,a)}{8}\left[ \left( \frac{2\left\vert f^{\prime
}(a)\right\vert ^{q}+\left\vert f^{\prime }(b)\right\vert ^{q}}{3}\right) ^{%
\frac{1}{q}}+\left( \frac{\left\vert f^{\prime }(a)\right\vert
^{q}+2\left\vert f^{\prime }(b)\right\vert ^{q}}{3}\right) ^{\frac{1}{q}}%
\right] .  \notag
\end{eqnarray}

\begin{proof}
Suppose that $a,a+\eta (b,a)\in K$. By assumption, using the well known
power mean inequality and (\ref{5}) in the proof of Theorem \ref{tt3}, we
have%
\begin{eqnarray*}
&&\left\vert \frac{1}{\eta (b,a)}\int_{a}^{a+\eta (b,a)}f(x)dx-f\left( \frac{%
2a+\eta (b,a)}{2}\right) \right\vert  \\
&\leq &\eta (b,a)\left[ \int_{0}^{\frac{1}{2}}t\left\vert f^{\prime
}(a+t\eta (b,a))\right\vert dt+\int_{\frac{1}{2}}^{1}(1-t)\left\vert
f^{\prime }(a+t\eta (b,a))\right\vert dt\right]  \\
&\leq &\eta (b,a)\left[ \left( \int_{0}^{\frac{1}{2}}tdt\right) ^{\frac{1}{p}%
}\left( \int_{0}^{\frac{1}{2}}t\left\vert f^{\prime }(a+t\eta
(b,a))\right\vert ^{q}dt\right) ^{\frac{1}{q}}\right.  \\
&&\left. +\left( \int_{\frac{1}{2}}^{1}\left( 1-t\right) dt\right) ^{\frac{1%
}{p}}\left( \int_{\frac{1}{2}}^{1}\left( 1-t\right) \left\vert f^{\prime
}(a+t\eta (b,a))\right\vert ^{q}dt\right) ^{\frac{1}{q}}\right]  \\
&\leq &\frac{\eta (b,a)}{8^{\frac{1}{p}}}\left[ \left( \int_{0}^{\frac{1}{2}%
}t\left[ (1-t)\left\vert f^{\prime }(a)\right\vert ^{q}+t\left\vert
f^{\prime }(b)\right\vert ^{q}\right] dt\right) ^{\frac{1}{q}}\right.  \\
&&\left. +\left( \int_{\frac{1}{2}}^{1}\left( 1-t\right) \left[
(1-t)\left\vert f^{\prime }(a)\right\vert ^{q}+t\left\vert f^{\prime
}(b)\right\vert ^{q}\right] dt\right) ^{\frac{1}{q}}\right]  \\
&=&\frac{\eta (b,a)}{8}\left[ \left( \frac{2\left\vert f^{\prime
}(a)\right\vert ^{q}+\left\vert f^{\prime }(b)\right\vert ^{q}}{3}\right) ^{%
\frac{1}{q}}+\left( \frac{\left\vert f^{\prime }(a)\right\vert
^{q}+2\left\vert f^{\prime }(b)\right\vert ^{q}}{3}\right) ^{\frac{1}{q}}%
\right] ,
\end{eqnarray*}%
where $\frac{1}{p}+\frac{1}{q}=1.$ The proof is completed.
\end{proof}
\end{theorem}

\begin{theorem}
\label{tt7} Under the assumptions of Theorem \ref{tt6}. Then the following
inequality holds:%
\begin{equation}
\left\vert \frac{1}{\eta (b,a)}\int_{a}^{a+\eta (b,a)}f(x)dx-f\left( \frac{%
2a+\eta (b,a)}{2}\right) \right\vert \leq \frac{\eta (b,a)}{8}(\frac{2^{%
\frac{1}{q}}+1}{3^{\frac{1}{q}}})\left[ \left\vert f^{\prime }(a)\right\vert
+\left\vert f^{\prime }(b)\right\vert \right]  \label{9}
\end{equation}

\begin{proof}
We consider the inequality (\ref{8}), i.e.%
\begin{eqnarray*}
&&\left\vert \frac{1}{\eta (b,a)}\int_{a}^{a+\eta (b,a)}f(x)dx-f\left( \frac{%
2a+\eta (b,a)}{2}\right) \right\vert \\
&& \\
&\leq &\frac{\eta (b,a)}{8}\left[ \left( \frac{2\left\vert f^{\prime
}(a)\right\vert ^{q}+\left\vert f^{\prime }(b)\right\vert ^{q}}{3}\right) ^{%
\frac{1}{q}}+\left( \frac{\left\vert f^{\prime }(a)\right\vert
^{q}+2\left\vert f^{\prime }(b)\right\vert ^{q}}{3}\right) ^{\frac{1}{q}}%
\right] .
\end{eqnarray*}%
Let $a_{1}=2\left\vert f^{\prime }(a)\right\vert ^{q}/3$, $b_{1}=\left\vert
f^{\prime }(b)\right\vert ^{q}/3$, $a_{2}=\left\vert f^{\prime
}(a)\right\vert ^{q}/3$, $b_{2}=2\left\vert f^{\prime }(b)\right\vert ^{q}/3$%
. Here $0<1/q<1$, for $q\geq 1$. Using the fact that%
\begin{equation*}
\dsum\limits_{k=1}^{n}\left( a_{k}+b_{k}\right) ^{s}\leq
\sum_{k=1}^{n}a_{k}^{s}+\sum_{k=1}^{n}b_{k}^{s}.
\end{equation*}%
For $\left( 0\leq s<1\right) $, $a_{1},a_{2},...,a_{n}\geq 0$, $%
b_{1},b_{2},...,b_{n}\geq 0$, we obtain%
\begin{eqnarray*}
&&\frac{\eta (b,a)}{8}\left[ \left( \frac{2\left\vert f^{\prime
}(a)\right\vert ^{q}+\left\vert f^{\prime }(b)\right\vert ^{q}}{3}\right) ^{%
\frac{1}{q}}+\left( \frac{\left\vert f^{\prime }(a)\right\vert
^{q}+2\left\vert f^{\prime }(b)\right\vert ^{q}}{3}\right) ^{\frac{1}{q}}%
\right] \\
&\leq &\frac{\eta (b,a)}{8}(\frac{2^{\frac{1}{q}}+1}{3^{\frac{1}{q}}})\left[
\left\vert f^{\prime }(a)\right\vert +\left\vert f^{\prime }(b)\right\vert %
\right] .
\end{eqnarray*}
\end{proof}
\end{theorem}

\section{Hermite-Hadamard type inequalities for log-preinvex function}

In this section, we shall continue with the following refinements of the
Hermite-Hadamard inequality for log-preinvex functions and we give some
results connected with the left part of (\ref{p}):

\begin{theorem}
\label{z} Let $K\subseteq 
\mathbb{R}
$ be an open invex subset with respect to $\eta :K\times K\rightarrow 
\mathbb{R}
$. Suppose that $f:K\rightarrow 
\mathbb{R}
$ is a diferentiable function. If $\left\vert f^{\prime }\right\vert $ is
log-preinvex on $K$ then, for every $a,b\in K$ the following inequality holds%
\begin{equation*}
\left\vert \frac{1}{\eta (b,a)}\int_{a}^{a+\eta (b,a)}f(x)dx-f\left( \frac{%
2a+\eta (b,a)}{2}\right) \right\vert \leq \eta (b,a)\left( \frac{\left\vert
f^{\prime }(b)\right\vert ^{\frac{1}{2}}-\left\vert f^{\prime
}(a)\right\vert ^{\frac{1}{2}}}{\log \left\vert f^{\prime }(b)\right\vert
-\log \left\vert f^{\prime }(a)\right\vert }\right) ^{2}
\end{equation*}
\end{theorem}

\begin{proof}
Suppose that $a,a+\eta (b,a)\in K$. By assumption and (\ref{5}) in the proof
of Theorem \ref{tt3}, integrating by parts implies that%
\begin{eqnarray*}
&&\left\vert \frac{1}{\eta (b,a)}\int_{a}^{a+\eta (b,a)}f(x)dx-f\left( \frac{%
2a+\eta (b,a)}{2}\right) \right\vert \\
&\leq &\eta (b,a)\left[ \int_{0}^{\frac{1}{2}}t\left\vert f^{\prime
}(a+t\eta (b,a))\right\vert dt+\int_{\frac{1}{2}}^{1}(1-t)\left\vert
f^{\prime }(a+t\eta (b,a))\right\vert dt\right] \\
&\leq &\eta (b,a)\left[ \int_{0}^{\frac{1}{2}}t\left\vert f^{\prime
}(a)\right\vert ^{1-t}\left\vert f^{\prime }(b)\right\vert ^{t}dt+\int_{%
\frac{1}{2}}^{1}(1-t)\left\vert f^{\prime }(a)\right\vert ^{1-t}\left\vert
f^{\prime }(b)\right\vert ^{t}dt\right] \\
&=&\eta (b,a)\left[ \int_{0}^{\frac{1}{2}}\left\vert f^{\prime
}(a)\right\vert t\left( \frac{\left\vert f^{\prime }(b)\right\vert }{%
\left\vert f^{\prime }(a)\right\vert }\right) ^{t}dt+\int_{\frac{1}{2}%
}^{1}(1-t)\left\vert f^{\prime }(b)\right\vert \left( \frac{\left\vert
f^{\prime }(b)\right\vert }{\left\vert f^{\prime }(a)\right\vert }\right)
^{1-t}dt\right] \\
&=&\eta (b,a)\left[ \frac{\left\vert f^{\prime }(a)\right\vert }{\log
\left\vert f^{\prime }(b)\right\vert -\log \left\vert f^{\prime
}(a)\right\vert }\left[ -\frac{1}{\log \left\vert f^{\prime }(b)\right\vert
-\log \left\vert f^{\prime }(a)\right\vert }\left( \frac{\left\vert
f^{\prime }(b)\right\vert }{\left\vert f^{\prime }(a)\right\vert }\right)
^{t}\right] _{0}^{\frac{1}{2}}\right. \\
&&\left. +\left[ \frac{1}{\log \left\vert f^{\prime }(b)\right\vert -\log
\left\vert f^{\prime }(a)\right\vert }\left( \frac{\left\vert f^{\prime
}(b)\right\vert }{\left\vert f^{\prime }(a)\right\vert }\right) ^{t}\right]
_{\frac{1}{2}}^{1}\right] \\
&=&\eta (b,a)\left[ \frac{-2\left\vert f^{\prime }(a)\right\vert ^{\frac{1}{2%
}}\left\vert f^{\prime }(b)\right\vert ^{\frac{1}{2}}}{\left( \log
\left\vert f^{\prime }(b)\right\vert -\log \left\vert f^{\prime
}(a)\right\vert \right) ^{2}}+\frac{\left\vert f^{\prime }(a)\right\vert }{%
\left( \log \left\vert f^{\prime }(b)\right\vert -\log \left\vert f^{\prime
}(a)\right\vert \right) ^{2}}\right. \\
&&\left. +\frac{\left\vert f^{\prime }(a)\right\vert }{\left( \log
\left\vert f^{\prime }(b)\right\vert -\log \left\vert f^{\prime
}(a)\right\vert \right) ^{2}}\right] \\
&=&\eta (b,a)\left[ \frac{\left\vert f^{\prime }(b)\right\vert ^{\frac{1}{2}%
}-\left\vert f^{\prime }(a)\right\vert ^{\frac{1}{2}}}{\log \left\vert
f^{\prime }(b)\right\vert -\log \left\vert f^{\prime }(a)\right\vert }\right]
^{2}
\end{eqnarray*}%
which completes the proof.
\end{proof}

\begin{theorem}
\label{fd} Let $K\subseteq 
\mathbb{R}
$ be an open invex subset with respect to $\eta :K\times K\rightarrow 
\mathbb{R}
$. Suppose that $f:K\rightarrow 
\mathbb{R}
$ is a diferentiable function. Assume $q\in 
\mathbb{R}
$ with $q\geq 1$. If $\left\vert f^{\prime }\right\vert ^{q}$ is
log-preinvex on $K$ then, for every $a,b\in K$ the following inequality holds%
\begin{eqnarray*}
&&\left\vert \frac{1}{\eta (b,a)}\int_{a}^{a+\eta (b,a)}f(x)dx-f\left( \frac{%
2a+\eta (b,a)}{2}\right) \right\vert \\
&\leq &\eta (b,a)\left[ \frac{\left\vert f^{\prime }(a)\right\vert ^{\frac{1%
}{2}}}{2^{\frac{1}{p}}\left( p+1\right) ^{\frac{1}{p}}q^{\frac{1}{q}}}\left( 
\frac{\left\vert f^{\prime }(b)\right\vert ^{\frac{q}{2}}-\left\vert
f^{\prime }(a)\right\vert ^{\frac{q}{2}}}{\log \left\vert f^{\prime
}(b)\right\vert -\log \left\vert f^{\prime }(a)\right\vert }\right) ^{\frac{1%
}{q}}\right] .
\end{eqnarray*}
\end{theorem}

\begin{proof}
By H\"{o}lder inequality and (\ref{5}) in the proof of Theorem \ref{tt3}, we
have 
\begin{eqnarray*}
&&\left\vert \frac{1}{\eta (b,a)}\int_{a}^{\eta (b,a)}f(x)dx-f\left( \frac{%
2a+\eta (b,a)}{2}\right) \right\vert \\
&\leq &\eta (b,a)\left[ \int_{0}^{\frac{1}{2}}t\left\vert f^{\prime
}(a+t\eta (b,a))\right\vert dt+\int_{\frac{1}{2}}^{1}(1-t)\left\vert
f^{\prime }(a+t\eta (b,a))\right\vert dt\right] \\
&\leq &\eta (b,a)\left[ \left( \int_{0}^{\frac{1}{2}}t^{p}dt\right) ^{\frac{1%
}{p}}\left( \int_{0}^{\frac{1}{2}}\left\vert f^{\prime }(a+t\eta
(b,a))\right\vert ^{q}\right) ^{\frac{1}{q}}dt+\left( \int_{\frac{1}{2}%
}^{1}(1-t)^{p}\right) ^{\frac{1}{p}}\left( \int_{\frac{1}{2}}^{1}\left\vert
f^{\prime }(a+t\eta (b,a))\right\vert ^{q}dt\right) ^{\frac{1}{q}}\right] \\
&\leq &\eta (b,a)\left[ \left( \int_{0}^{\frac{1}{2}}t^{p}dt\right) ^{\frac{1%
}{p}}\left( \int_{0}^{\frac{1}{2}}\left( \left\vert f^{\prime
}(a)\right\vert ^{1-t}\left\vert f^{\prime }(b)\right\vert ^{t}\right)
^{q}dt\right) ^{\frac{1}{q}}\right. \\
&&\left. +\left( \int_{\frac{1}{2}}^{1}(1-t)^{p}\right) ^{\frac{1}{p}}\left(
\int_{\frac{1}{2}}^{1}\left( \left\vert f^{\prime }(a)\right\vert
^{1-t}\left\vert f^{\prime }(b)\right\vert ^{t}\right) ^{q}dt\right) ^{\frac{%
1}{q}}\right] \\
&=&\eta (b,a)\left[ \frac{\left\vert f^{\prime }(a)\right\vert ^{\frac{1}{2}}%
}{2^{\frac{1}{p}}\left( p+1\right) ^{\frac{1}{p}}q^{\frac{1}{q}}}\left( 
\frac{\left\vert f^{\prime }(b)\right\vert ^{\frac{q}{2}}-\left\vert
f^{\prime }(a)\right\vert ^{\frac{q}{2}}}{\log \left\vert f^{\prime
}(b)\right\vert -\log \left\vert f^{\prime }(a)\right\vert }\right) ^{\frac{1%
}{q}}\right]
\end{eqnarray*}%
where $\frac{1}{p}+\frac{1}{q}=1.$
\end{proof}

Now, we give the followig results connected with the left  part of (\ref{h})
for classical log-convex functions.

\begin{corollary}
\label{q} Under the assumptions of Theorem \ref{z} with $\eta (b,a)=b-a,$
then the following inequality holds:%
\begin{equation*}
\left\vert \frac{1}{b-a}\int_{a}^{b}f(x)dx-f\left( \frac{a+b}{2}\right)
\right\vert \leq (b-a)\left( \frac{\left\vert f^{\prime }(b)\right\vert ^{%
\frac{1}{2}}-\left\vert f^{\prime }(a)\right\vert ^{\frac{1}{2}}}{\log
\left\vert f^{\prime }(b)\right\vert -\log \left\vert f^{\prime
}(a)\right\vert }\right) ^{2}.
\end{equation*}
\end{corollary}

\begin{corollary}
\label{q1} Under the assumptions of Theorem \ref{fd}\ with $\eta (b,a)=b-a,$
then the following inequality holds:%
\begin{eqnarray*}
&&\left\vert \frac{1}{b-a}\int_{a}^{b}f(x)dx-f\left( \frac{a+b}{2}\right)
\right\vert \\
&\leq &(b-a)\left[ \frac{\left\vert f^{\prime }(a)\right\vert ^{\frac{1}{2}}%
}{2^{\frac{1}{p}}\left( p+1\right) ^{\frac{1}{p}}q^{\frac{1}{q}}}\left( 
\frac{\left\vert f^{\prime }(b)\right\vert ^{\frac{q}{2}}-\left\vert
f^{\prime }(a)\right\vert ^{\frac{q}{2}}}{\log \left\vert f^{\prime
}(b)\right\vert -\log \left\vert f^{\prime }(a)\right\vert }\right) ^{\frac{1%
}{q}}\right] .
\end{eqnarray*}
\end{corollary}

\section{An extension to several variables functions}

In this section, we shall extend the Corollary \ref{q} and Corollary \ref{q1}
to functions of several variables defined on invex subsets of $%
\mathbb{R}
^{n}$

Let $K\subseteq 
\mathbb{R}
^{n}$ be an invex set with respect to $\eta :K\times K\rightarrow 
\mathbb{R}
^{n}.$ For every $x,y\in K$ the $\eta $-path $P_{xv}$ joining the points $x$
and $v:=x+\eta (y,x)$ is defined as follows%
\begin{equation*}
P_{xv}=\left\{ z:z=x+t\eta (y,x):t\in \left[ 0,1\right] \right\} .
\end{equation*}

\begin{proposition}
\label{eqw}Let $K\subseteq 
\mathbb{R}
^{n}$ be an invex set with respect to $\eta :K\times K\rightarrow 
\mathbb{R}
^{n}$ and $f:K\rightarrow 
\mathbb{R}
$ is a functio. Suppose that $\eta $ satisfies Condition C on $K.$ Then for
every $x,y\in K$ the function $f$ is log-preinvex with respect to $\eta $ on 
$\eta $-path $P_{xv}$ if and only if the function $\varphi :\left[ 0,1\right]
\rightarrow 
\mathbb{R}
$ defined by%
\begin{equation*}
\varphi \left( t\right) :=f\left( x+t\eta (y,x)\right) ,
\end{equation*}%
is log-convex on $\left[ 0,1\right] .$
\end{proposition}

\begin{proof}
Suppose that $\varphi $ is log-convex on $\left[ 0,1\right] $ and $%
z_{1}:=x+t_{1}\eta (y,x)\in P_{xv}$, $z_{2}:=x+t_{2}\eta (y,x)\in P_{xv}$.
Fix $\lambda \in \left[ 0,1\right] $. By (\ref{c}), we have%
\begin{eqnarray*}
f\left( z_{1}+\lambda \eta (z_{2},z_{1})\right) &=&f\left( x+\left( \left(
1-\lambda \right) t_{1}+\lambda t_{2}\right) \eta (y,x)\right) \\
&=&\varphi \left( \left( 1-\lambda \right) t_{1}+\lambda t_{2}\right) \\
&\leq &\left[ \varphi \left( t_{1}\right) \right] ^{\left( 1-\lambda \right)
}\left[ \varphi \left( t_{2}\right) \right] ^{\lambda } \\
&=&\left[ f\left( z_{1}\right) \right] ^{\left( 1-\lambda \right) }\left[
f\left( z_{2}\right) \right] ^{\lambda }
\end{eqnarray*}%
Hence, $f$ is log-preinvex with respect to $\eta $ on $\eta $-path $P_{xv}.$

Conversely, let $x,y\in K$ and the function $f$ be log-preinvex with respect
to $\eta $ on $\eta $-path $P_{xv}$. Suppose that $t_{1},t_{2}\in \left[ 0,1%
\right] $. Then, for every $\lambda \in \left[ 0,1\right] $ we have%
\begin{eqnarray*}
\varphi \left( \left( 1-\lambda \right) t_{1}+\lambda t_{2}\right)
&=&f\left( x+\left( \left( 1-\lambda \right) t_{1}+\lambda t_{2}\right) \eta
(y,x)\right) \\
&=&f\left( x+t_{1}\eta (y,x)+\lambda \eta (x+t_{2}\eta (y,x),x+t_{1}\eta
(y,x))\right) \\
&\leq &\left[ f\left( x+t_{1}\eta (y,x)\right) \right] ^{\left( 1-\lambda
\right) }\left[ f\left( x+t_{2}\eta (y,x)\right) \right] ^{\lambda } \\
&=&\left[ \varphi \left( t_{1}\right) \right] ^{\left( 1-\lambda \right) }%
\left[ \varphi \left( t_{2}\right) \right] ^{\lambda }
\end{eqnarray*}%
Therefore, $\varphi $ is log-convex on $\left[ 0,1\right] $.
\end{proof}

The following Teorem is a generalization of Corollary \ref{q}.

\begin{theorem}
\label{qw} Let $K\subseteq 
\mathbb{R}
^{n}$ be an invex set with respect to $\eta :K\times K\rightarrow 
\mathbb{R}
^{n}$ and $f:K\rightarrow 
\mathbb{R}
^{+}$ is a function.\ Suppose that $\eta $ satisfies Condition C on $K.$
Then for every $x,y\in K$ the function $f$ is log-preinvex with respect to $%
\eta $ on $\eta $-path $P_{xv}$. Then, for every $a,b\in \left( 0,1\right) $
with $a<b$ the following inequality holds%
\begin{eqnarray}
&&\left\vert \frac{1}{b-a}\int_{a}^{b}\left( \int_{0}^{s}f\left( x+t\eta
(y,x)\right) dt\right) ds-\int_{0}^{\frac{a+b}{2}}f\left( x+s\eta
(y,x)\right) ds\right\vert  \label{1} \\
&\leq &\left( b-a\right) \left[ \frac{\left[ f\left( x+b\eta (y,x)\right) %
\right] ^{\frac{1}{2}}-\left[ f\left( x+a\eta (y,x)\right) \right] ^{\frac{1%
}{2}}}{\log f\left( x+b\eta (y,x)\right) -\log f\left( x+a\eta (y,x)\right) }%
\right] ^{2}.  \notag
\end{eqnarray}

\begin{proof}
Let $x,y\in K$ and $a,b\in \left( 0,1\right) $ with $a<b$. Since $f$ is
log-preinvex with respect to $\eta $ on $\eta $-path $P_{xv}$ by Proposition %
\ref{eqw} the function $\varphi :\left[ 0,1\right] \rightarrow 
\mathbb{R}
^{+}$ defined by%
\begin{equation*}
\varphi \left( t\right) :=f\left( x+t\eta (y,x)\right) ,
\end{equation*}%
is log-convex on $\left[ 0,1\right] $. Now, we define the function $\phi :%
\left[ 0,1\right] \rightarrow 
\mathbb{R}
^{+}$ as follows%
\begin{equation*}
\phi \left( t\right) :=\int_{0}^{t}\varphi \left( s\right)
ds=\int_{0}^{t}f\left( x+s\eta (y,x)\right) ds.
\end{equation*}%
Obviously for every $t\in \left( 0,1\right) $ we have%
\begin{equation*}
\phi ^{\prime }\left( t\right) =\varphi \left( t\right) =f\left( x+t\eta
(y,x)\right) \geq 0
\end{equation*}%
hence, $\left\vert \phi ^{\prime }\left( t\right) \right\vert =\phi ^{\prime
}\left( t\right) $. Applying Corollary \ref{q} to the function $\phi $
implies that%
\begin{equation*}
\left\vert \frac{1}{b-a}\int_{a}^{b}\phi \left( t\right) dt-\phi \left( 
\frac{a+b}{2}\right) \right\vert \leq \left( b-a\right) \left( \frac{%
\left\vert \phi ^{\prime }(b)\right\vert ^{\frac{1}{2}}-\left\vert \phi
^{\prime }(a)\right\vert ^{\frac{1}{2}}}{\log \left\vert \phi ^{\prime
}(b)\right\vert -\log \left\vert \phi ^{\prime }(a)\right\vert }\right) ^{2}
\end{equation*}%
and we deduce that (\ref{1}) holds.
\end{proof}
\end{theorem}

\begin{remark}
\label{ws} Let $\varphi \left( t\right) :\left[ 0,1\right] \rightarrow 
\mathbb{R}
^{+}$ be a function and $q$ a positive real number, then $\varphi $ is
log-convex if and only if the function $\varphi ^{q}\left( t\right) :\left[
0,1\right] \rightarrow 
\mathbb{R}
^{+}$ is log-convex. \.{I}ndeed for every $x,y\in \left[ 0,1\right] $ it is
easy to see that%
\begin{equation*}
\left[ \left[ \varphi \left( x\right) \right] ^{1-t}\left[ \varphi \left(
y\right) \right] ^{t}\right] ^{q}=\left[ \varphi ^{q}\left( x\right) \right]
^{1-t}\left[ \varphi ^{q}\left( y\right) \right] ^{t}
\end{equation*}%
Therefore if $t\in \left[ 0,1\right] ,$ we have

$\varphi \left( tx+(1-t)y\right) \leq \left[ \varphi \left( x\right) \right]
^{1-t}\left[ \varphi \left( y\right) \right] ^{t}$ if and only if $\varphi
^{q}\left( tx+(1-t)y\right) \leq \left[ \varphi ^{q}\left( x\right) \right]
^{1-t}\left[ \varphi ^{q}\left( y\right) \right] ^{t}$.
\end{remark}

The following Theorem is a generalization Corollory \ref{q1} to functions
several variables.

\begin{theorem}
Let $K\subseteq 
\mathbb{R}
^{n}$ be an invex set with respect to $\eta :K\times K\rightarrow 
\mathbb{R}
^{n}$ and $f:K\rightarrow 
\mathbb{R}
^{+}$ is a function.\ Suppose that $\eta $ satisfies condition C on $K.$
Then for every $x,y\in K$ the function $f$ is log-preinvex with respect to $%
\eta $ on $\eta $-path $P_{xv}$. Then, for every $p>1$ and $a,b\in \left(
0,1\right) $ with $a<b$ the following inequality holds%
\begin{eqnarray}
&&\left\vert \frac{1}{b-a}\int_{a}^{b}\left( \int_{0}^{s}f\left( x+t\eta
(y,x)\right) dt\right) ds-\int_{0}^{\frac{a+b}{2}}f\left( x+s\eta
(y,x)\right) ds\right\vert  \label{2} \\
&\leq &(b-a)\left[ \frac{\left[ f\left( x+a\eta (y,x)\right) \right] ^{\frac{%
1}{2}}}{2^{\frac{1}{p}}\left( p+1\right) ^{\frac{1}{p}}q^{\frac{1}{q}}}%
\left( \frac{\left[ f\left( x+b\eta (y,x)\right) \right] ^{\frac{q}{2}}-%
\left[ f\left( x+a\eta (y,x)\right) \right] ^{\frac{q}{2}}}{\log f\left(
x+b\eta (y,x)\right) -\log f\left( x+a\eta (y,x)\right) }\right) ^{\frac{1}{q%
}}\right]  \notag
\end{eqnarray}

where $\frac{1}{p}+\frac{1}{q}=1$.

\begin{proof}
Let $x,y\in K$ and $a,b\in \left( 0,1\right) $ with $a<b$.\ Suppose that $%
\phi $ and $\varphi $ are the functions whixh are defined in the Theorem \ref%
{qw}. Since $\left\vert \phi ^{\prime }\right\vert :\left[ 0,1\right]
\rightarrow 
\mathbb{R}
^{+}$ is log-convex on $\left[ 0,1\right] $, by Remark \ref{ws} the function 
$\left\vert \phi ^{\prime }\right\vert ^{q}$ is also is log-convex on $\left[
0,1\right] $. Now, by applying Corollary \ref{q1} to function $\phi $ we get%
\begin{eqnarray*}
&&\left\vert \frac{1}{b-a}\int_{a}^{b}\phi (x)dx-\phi \left( \frac{a+b}{2}%
\right) \right\vert \\
&\leq &(b-a)\left[ \frac{\left\vert \phi ^{\prime }(a)\right\vert ^{\frac{1}{%
2}}}{2^{\frac{1}{p}}\left( p+1\right) ^{\frac{1}{p}}q^{\frac{1}{q}}}\left( 
\frac{\left\vert \phi ^{\prime }(b)\right\vert ^{\frac{q}{2}}-\left\vert
\phi ^{\prime }(a)\right\vert ^{\frac{q}{2}}}{\log \left\vert \phi ^{\prime
}(b)\right\vert -\log \left\vert \phi ^{\prime }(a)\right\vert }\right) ^{%
\frac{1}{q}}\right]
\end{eqnarray*}%
and we deduce that (\ref{2}) holds. The proof is complete.
\end{proof}
\end{theorem}

\end{document}